\documentclass[11pt, a4paper]{article}
\usepackage[latin1]{inputenc}
\usepackage{algorithm}
\usepackage{algorithmic} 
\usepackage[Sonny]{fncychap} 
\usepackage[hyphens]{url} 
\usepackage{tabularx}
\usepackage{appendix} 
\usepackage{enumitem}

\usepackage[hidelinks]{hyperref}

\usepackage[a4paper]{geometry}
\usepackage{amsfonts}
\usepackage{amsthm}
\usepackage{amsmath}
\usepackage{bbm}
\usepackage{parskip}
\usepackage{mathrsfs}
\usepackage{graphicx}
\usepackage{amssymb}
\usepackage{xtab}
\usepackage{setspace}
\usepackage{fancyhdr} 
\usepackage[dvipsname]{xcolor}
\usepackage{color,soul}
\usepackage[nodayofweek]{datetime}
\usepackage{hyperref}
\usepackage{mathtools}
\newdateformat{monthyear}{\monthname[\THEMONTH], \THEYEAR}
\newcommand{\N}{\mathbb{N}}

\newcommand{\Z}{\mathbb{Z}}
\newcommand{\E}{\mathbb{E}}
\newcommand{\R}{\mathbb{R}}
\newcommand{\p}{\mathbb{P}}

\newcommand{\Partial}[2]{\frac{\partial #1}{\partial #2}}

\newcommand{\scndmxPartial}[3]{\frac{\partial^2 #1}{\partial #2 \partial #3}}

\newtheorem{proposition}{Proposition}[section]
\newtheorem{lemma}[proposition]{Lemma}

\newtheorem{theorem}[proposition]{Theorem}

\newtheorem{corollary}[proposition]{Corollary}

\renewcommand{\epsilon}{\varepsilon}

\numberwithin{equation}{section}

\title{At the Edge of a Cloud of Brownian Particles}
\author{Dom Brockington\footnote{Mathematics Institute, University of Warwick, Coventry CV4 7AL, UK. Email Address: Dominic.Brockington@warwick.ac.uk.} and Jon Warren\footnote{Department of Statistics, University of Warwick, Coventry CV4 7AL, UK. Email Address: J.Warren@warwick.ac.uk}}
\begin{document}
\maketitle

\begin{abstract}
    We study a simple model for the trajectory of a particle in a turbulent fluid, where a Brownian motion travels through a random Gaussian velocity field. We study the quenched law of the process and prove that in a weak environment setting, the fluctuations at the transition from weak to strong disorder are described by the KPZ equation. We conjecture the same is true without the assumption of a weak environment.
\end{abstract}

\section{Introduction}
In this paper we consider a model of turbulent advection describing particle trajectories in one dimensional turbulent fluids. The particles each have their own, independent molecular diffusive motion, and are  also transported through a Gaussian random drift field which represents the effect of the fluid. The field  is Brownian in time and smooth, but rapidly decorrelating, in space. Associated to this model is a stochastic flow of kernels, \cite{lejan2004}, which can be thought of as the density of a cloud of particles in the fluid, or as the random transition density of a single particle running through a realisation of the drift field. Using this interpretation, it is possible to show the flow of kernels solve a stochastic partial differential equation (SPDE) similar to a Fokker-Planck equation for a Brownian motion with drift. 

The model is an example of the compressible Kraichnan model for turbulence, see the review \cite{ParticlesandFieldsInFluidTurbulence}, which models the trajectory of a particle in a turbulent fluid, where the fluid is represented by a random evolving velocity field that is white in time. In our case the spatial correlations are taken to be of short length and smooth in space, similar to the case considered in \cite{Gawedzki2004}, where the authors showed that removing the molecular diffusivity, at the same time as reducing the correlation length of the velocity field, with a certain relation holding  between the two, led to sticky interactions between pairs of particles in the limiting process. This result was extended for the model we consider, \cite{Warren2015}, to a full description of the interactions between any number of particles in the limiting process. 

In \cite{QuenchedLocalLimitStochasticFlows} Dunlap and Gu  show that the fluctuations of the density of the flow of kernels are well approximated, at large times, by the product of the heat kernel and the stationary solution to the SPDE solved by the flow of kernels. We are instead interested in the fluctuations of the density in the tail. Barraquand and Le Doussal, \cite{Barraquand_2020} conjecture that, at a cetain distance from the origin, these fluctuations are governed by the Kardar-Parisi-Zhang equation (KPZ equation). We will show that this is indeed the case in a regime in which the strength of the environment is tending to zero.  Further we conjecture that the KPZ equation also appears when the  diffusivity of the molecular motion of the particles is  taken to be small instead of the environment strength.

We now briefly recall  the stochastic heat equation, its connection to the KPZ equation and describe  some previous results on the universality of the KPZ equation. The stochastic heat equation (SHE) is the stochastic partial differential equation, driven by a space-time white noise $\dot{W}$, given below
\begin{align}\label{Intro: SHE}
    \partial_t z = \frac{\nu}{2}\Delta z + \kappa z \dot{W}.
\end{align}
The logarithm of the stochastic heat equation, $h= \frac{\nu}{2\delta} \log z$, is the Cole-Hopf solution to the KPZ equation
\begin{align}
    \partial_t h = \frac{\nu}{2} \Delta h + \delta (\partial_x h)^2 + \kappa \dot{W}.
\end{align}
The KPZ equation describes the interface for surface growth model and is the canonical model in the KPZ universality class \cite{KPZ}. In fact, the KPZ equation  is also a  universal object itself; this is known as weak KPZ universality and it has been shown that a large class of continuous surface growth models lie in this class \cite{hairer_quastel_2018}. In addition, it has been shown that certain observables of some discrete models converge to the KPZ equation, via convergence to the stochastic heat equation. The earliest such result was for the height function of the weakly asymmetric exclusion process \cite{StochasticBurgersandKPZFromParticleSystems}. More recently convergence to the KPZ equation has been shown for the free energy of directed random polymers in the intermediate disorder limit \cite{TheIntermediateDisorderRegimeForDirectedPolymers}. Following this result weak KPZ universality has been shown for a generalisation of ASEP \cite{ASEP(qj)toKPZ}, for a class of weakly asymmetric non-simple exclusion processes \cite{dembo2016weakly}, and for Higher-Spin Exclusion process \cite{KPZfromHSEP} and the related Stochastic 6-vertex model \cite{KPZfromS6V}. Most relevant for us, in \cite{KPZforRWRE} Corwin and Gu studied a discrete version of our model, where the random drift field appears instead as random transition probabilities for a nearest neighbour random walk on $\Z$. They showed that the transition probabilities evaluated in the large deviation regime, after rescaling, converge to the solution to the stochastic heat equation.


 Barraquand and Corwin, \cite{Barraquand2017}, showed that when the transition probabilities in the random walk model are chosen to be Beta distributed the model becomes exactly solvable. They then used their formulae to show that the tail probabilities of the Beta Random walk in a random environment have Tracy-Widom GUE fluctuations of size $N^{1/3}$, placing the model in the universality class of the KPZ fixed point, called strong KPZ universality. The same result is expected to hold for the density of the transition probabilities evaluated at point in the tail, not just for the cumulative tail probabilities \cite{ExactSolutionRWRE1dThieryLeDoussal}.

Barraquand and Rychnovsky, \cite{barraquand2019large}, found similar results exist for sticky Brownian motions and the associated Howitt-Warren flows; once again taking advantage of exact solvability in a special case. Further they conjectured the Howitt-Warren flows (  \cite{HowittWarren}) converge, under rescaling, to the stochastic heat equation, based on the convergence of the moments. Note that by \cite{Warren2015} this special case  of the sticky  flow is not the same one that arises in the limit of our model, but it is reasonable to expect the results to hold more generally.   

In this paper we show that the flow of kernels associated with our model converges to the solution to the stochastic heat equation, when the strength of the random environment  and its correlation length are both taken to zero, with a certain relationship holding between the two.  This result is similar to the result of Corwin and Gu, \cite{KPZforRWRE}, for the discrete random walk model. We will also consider the behaviour of the  second moment of the kernel in  more general regimes; Figure \ref{diagram} shows the behavior we find to hold when  the distance from  the origin and the ratio of the strength of environment to molecular diffusivity  vary.  The stochastic  heat equation occurs (or is conjectured to occur) at a transition between weak  and strong disorder  regimes. In the former the kernels converge to the deterministic heat equation. In the latter the kernels converge to zero.  This is analogous  to  the behaviour of the partition function of a random polymer, \cite{Comets}. 



\subsection{The Model}\label{The model}
Suppose that $W_C$ is a centred Gaussian field with covariance $\E[ W_C(s, x) W_C(t, y) ] = (s\wedge t) C(x-y)$, where $C\in C^\infty_c(\R)$ is positive-definite. Let $\sigma>0$, we are interested in the solutions to the SDE
\begin{align}\label{system of SDEs}
    dX(t) = W_C(dt, X(t)) + \sigma dB(t).
\end{align}
In the above SDE, $B$ is a Brownian motion on $\R$ independent of $W_C$ and both stochastic integrals are understood in the It\^o sense. As mentioned in the introduction, this is a simplified model for a particle trajectory in a turbulent fluid. The fluid is represented by the noise, $W_C$, and $\sigma$ is the molecular diffusivity of the particle itself. Existence and uniqueness of solutions is proved in \cite[Theorem 3.4.1]{Kunita}. In fact, by \cite[Theorem 4.5.1]{Kunita} the solution exists as a stochastic flow, $X_{s,t}(x)$, such that for any $s\geq 0$, the process $(X_{s,t}(x))_{t>s}$ solves the SDE started from $x$ at time $s$.
The solution is distributed as a Brownian motion with diffusivity $\nu:=C(0) + \sigma^2$, which can be checked directly by calculating the quadratic variation and recalling L\'evy's characterisation of Brownian motion. Alternatively, the solution can be thought of as a Brownian motion with diffusivity $\sigma^2$, running through the time dependent Gaussian random field $W_C$, which we can think of as a random velocity field. We then consider the transition probabilities $(U_{s,t})_{s\leq t}$ of the process $(X_{s,t})_{s\leq t}$ conditional on $W_C$:
\begin{equation}\label{Def: Flow of kernels for turbulent SDE}
    U_{s,t}(x, A) := \p(X_{s,t}(x)\in A| \ W),
\end{equation}
Note that the family of kernels $(U_{s,t})_{s<t}$ depend only on the field $W$, and form a stochastic flow of kernels as introduced by Le Jan and Raimond in \cite{lejan2004}. Further, because $X$ is itself a Brownian motion, if we average $U_{s,t}$ over the law of $W$ we get the heat kernel,
\begin{align}\label{average of the flow of kernels}
    \E[U_{s,t}(x,A)] = P_{t-s}(x,A).
\end{align}
Here, $P$ denotes the heat kernel with diffusivity $\nu$, the density of which we will denote by $p_{t-s}(x-y)= \frac{1}{\sqrt{2\pi \nu (t-s)}} e^{-\frac{(x-y)^2}{2\nu (t-s)}}$.\par{}

It is known from \cite{QuenchedLocalLimitStochasticFlows} that the family of probability kernels $(U(s, t, x, dy))_{s< t}$ have continuous densities, $\mathcal{U}(s,t , x, \cdot)$,  with respect to the Lebesgue measure which solve the stochastic partial differential equations
\begin{align}
    \partial_t \mathcal{U} = \frac{\nu}{2} \partial^2_y \mathcal{U} - \partial_y\left(\mathcal{U} \dot{W}_C \right), \label{General SPDE}
\end{align}
together with the initial condition $\mathcal{U}(s,s, x, y) = \delta(x- y)$, where $\delta$ is the Dirac delta. In the above equation, $\dot{W}_C$ is the (formal) time derivative of $W_C$, so that it is white in time and smoothly correlated in space. By solution, we mean it is a generalised solution in the sense of \cite{Kunita1994}, which we describe now by briefly recalling \cite[Proposition 2.1]{QuenchedLocalLimitStochasticFlows}. The process $\mathcal{U}(0,t,x, \cdot)$, considered as a time-indexed family of tempered distributions on $\R$, is the unique solution to (\ref{General SPDE}). That is, for every $s, x\in\R$ and every Schwartz function $f:\R\to \R$ the following equality holds almost surely for every $t>s$
\begin{align}\label{Weak Formulation of u SPDE}
    &\int_\R \mathcal{U}(s,t, x, y) f(y) dy\nonumber\\
    = &f(x) + \frac{\nu}{2}\int_s^t \int_\R \mathcal{U}(s,r, x, y) f''(y)dy dr + \int_s^t \int_\R \mathcal{U}(s, r, x, y) f'(y) W_C(dr, y)dy.    
\end{align}
Here, the stochastic integral is interpreted in the It\^o sense.  \par{}

 The same SPDE, in a slightly different formulation, was derived for the flow of kernels (\ref{Def: Flow of kernels for turbulent SDE}) in \cite[Section 5]{lejan2004}. The solution to the SPDE can be constructed directly in terms of a Wiener chaos expansion, \cite[Theorem 3.2]{IntegrationOfBrownianVectorFields}. Note that (\ref{General SPDE}) is similar to the Fokker-Planck equation for a Brownian motion with diffusivity $\sigma^2$ moving through a velocity field $\dot{W}_C$; however, the coefficient of the Laplacian part is $\nu$ instead of $\sigma^2$ because of the It\^o correction coming from the stochastic integral.

\subsection{The Stochastic Heat Equation as a Limit}\label{The Stochastic Heat Equation as a Limit}

We are interested in the fluctuations of $\mathcal{U}$ at large times far from the origin. Hence, we look first at the effect of diffusive rescaling on the kernel. Define for $\epsilon>0$,
\begin{align*}
    \mathcal{U}^\epsilon(s,t,x,y)= \epsilon^{-1}\mathcal{U}(\epsilon^{-2}s, \epsilon^{-2}t, \epsilon^{-1}x, \epsilon^{-1}y).
\end{align*}
Then it is a straightforward calculation to show that $\mathcal{U}^\epsilon$ is a solution to the transport SPDE (\ref{Weak Formulation of u SPDE}) in which $W_C$ is replaced by its rescaled version $W^{\epsilon}(t,y)= \epsilon W_{C}(\epsilon^{-2}t, \epsilon^{-1}y)$, which has spatial covariance function $C^{\epsilon}(x-y)= C(\epsilon^{-1}(x-y))$. If we take $\epsilon\to 0$ then $\mathcal{U}^\epsilon$ converges to the solution to the Heat equation, as we will see in Proposition \ref{Weak convergence to the heat equation}. This convergence is only in a weak sense in the space variable; Dunlap and Gu, \cite{QuenchedLocalLimitStochasticFlows}, show that there is a non-trivial random fluctuation when the pointwise behaviour is considered. \par{}

Rather than just considering in the diffusive regime, we are also interested in the moderate deviation regime, where $y\approx t^{3/4}$ and Barraquand and Le Doussal, \cite{Barraquand_2020}, conjectured the SHE would appear. Thus, we introduce the rescaled quantity, $\mathcal{V}=\mathcal{V}^{\epsilon, \lambda}$, given by
\begin{align}\label{Def: Tilted kernel}
    \mathcal{V}(s, t, x, y) := e^{\frac{\nu }{2}\lambda^2 (t-s) + \lambda (y-x)} \mathcal{U}^\epsilon(s, t, x+\lambda \nu s, y+\lambda \nu  t)
\end{align}
where $\lambda \in \R$ is a parameter which we will vary with $\epsilon$ to control the distance from the origin. The prefactor is motivated by fixing the expectation of $\mathcal{V}$:
\begin{align*}
    \E\left[\mathcal{V}(s, t, x, y) \right] =& e^{\frac{\nu }{2}\lambda^2 t + \lambda (y-x)} \E\left[ \mathcal{U}^\epsilon(s, t, x, y+\lambda \nu  t)\right] \\
    =& e^{\frac{\nu }{2}\lambda^2 t - \lambda (y-x)} p_{t-s}(x-y)\\
    =& p_{t-s}(x-y).
\end{align*}
In the most part, we will set $s,x=0$ and consider $\mathcal{V}(t, y):= \mathcal{V}(0,t,0,y)$. According to Lemma \ref*{Lemma: SPDE for v} below, $\mathcal{V}$ is the solution to the SPDE,
\begin{equation}\label{SPDE for v}
    \partial_t \mathcal{V} =\frac{\nu}{2} \partial^2_y \mathcal{V} + \lambda \mathcal{V} \dot{W}^{\epsilon} - \partial_y\left( \mathcal{V} \dot{W}^{\epsilon} \right), \quad \mathcal{V}(0,y)= \delta_0(y),
\end{equation}
where, here, $\dot{W}^\epsilon$ is a shifted version of the noise, $W^\epsilon(t, y):= \epsilon \int_0^{\epsilon^{-2}t} W_C(ds, \epsilon^{-1}y+\lambda \epsilon^{-1} \nu s)$. The covariance of the shifted noise is given by $\E\left[ W^\epsilon(t,x) W^\epsilon(s,y) \right] = (s\wedge t) C^\epsilon(x-y)$, so that $W^\epsilon$ is equal in distribution to the $W^\epsilon$ discussed at the start of the section.\par{}
If $\lambda$ is chosen to depend on $\epsilon$ in an appropriate manner, namely $\lambda=\epsilon^{-\frac{1}{2}}$, then $\lambda^2 C^{(\epsilon)}$ approaches a multiple of the Dirac delta distribution, and thus $\dot{W}^{\epsilon}$ approaches a space-time white noise as $\epsilon\to 0$. Le Doussal and Barraquand, \cite{Barraquand_2020}, argue that taking $\epsilon\to 0$ in (\ref{SPDE for v}) will lead to the final term vanishing, suggesting the limit could be the stochastic heat equation.\par{}

One can also consider the moments of $\mathcal{V}$. First consider the moments of $\mathcal{U}$, which have a representation in terms of solutions to (\ref{system of SDEs}). More precisely, if $X^1$ and $X^2$ are two solutions to $(\ref*{system of SDEs})$, each driven by independent Brownian motions $B^1$ and $B^2$ but the same Gaussian velocity field, then for any $f\in C_b(\R^2)$
\begin{align*}
    \E\left[\int_{\R^2} \mathcal{U}(0,t,0,y_1)\mathcal{U}(0,t,0,y_2)f(y_1,y_2)dy\right] = \E\left[ f(X^1(t),X^2(t)) \right].
\end{align*}
Notice that $(X^1(t),X^2(t))_{t\geq 0}$ is a diffusion with generator
\begin{align}\label{X generator}
    \mathcal{G}f= \frac{1}{2}\sum_{i,j=1}^2 (\sigma^2 \delta_{i,j} + C^\epsilon(y_i-y_j))\scndmxPartial{f}{y_i}{y_j}.
\end{align}
To see that this representation is useful, we can note that it can be used to show the convergence towards the heat equation, in the weak spatial sense, simply by showing that $(\epsilon X^1(\epsilon^{-2}t), \epsilon X^2(\epsilon^{-2}t))_{t\geq 0}$ converges towards a Brownian motion on $\R^2$ with diffusivity $\nu$ as $\epsilon \to 0$; we will do this later in propositions \ref{Weak convergence to the heat equation} and \ref{convergence of the 2pt motion}. There is a similar representation for the moments of $\mathcal{V}$,
\begin{align}\label{early2pt}
    &\E\left[\int_{\R^2} \mathcal{V}(t,y_1)\mathcal{V}(t,y_2)f(y_1,y_2)dy \right] = \E\left[ e^{\lambda^2 \int_0^t C^\epsilon(Y^1(s)-Y^2(s))ds} f(Y^1(t),Y^2(t)) \right],
\end{align}
where $Y=(Y^1, Y^2)$ is a diffusion with generator 
\begin{align*}
    \mathcal{G}f= \frac{1}{2}\sum_{i,j=1}^2 (\sigma^2 \delta_{i,j} + C^\epsilon(y_i-y_j))\scndmxPartial{f}{y_i}{y_j} + \sum_{i=1}^2 \lambda C^\epsilon(y_1-y_2)\Partial{f}{y_i}, \quad \text{for } f\in C^2_b(\R^2),
\end{align*}
$\delta_{i,j}$ being the Kronecker delta.\par{}

By taking $\epsilon \to 0$, $Y^1$ and $Y^2$ once again converge to independent Brownian motions, and if we simultaneously take $\lambda=\epsilon^{-\frac{1}{2}}$, then the right hand side of (\ref{early2pt}) approaches 
\begin{align}\label{second moment limit}
    \E\left[ \exp\left(\frac{\kappa^2}{2 \nu} \mathcal{L}_t^0(B^1-B^2)\right) f(B^1(t),B^2(t)) \right],
\end{align}
where $B^1$ and $B^2$ are independent Brownian motions with diffusivity $\nu$, $\mathcal{L}^0_t(B^1-B^2)$ is the local time at $0$ of the process $B^1-B^2$ and $\kappa^2= \nu \int_\R\frac{C(y)}{\sigma^2+ C(0) - C(y)}dy$. This is exactly the form of the second moment for the stochastic heat equation, (\ref{Intro: SHE}). Note that, importantly, this prediction for the coefficient, $\kappa^2$, of the noise is different from that which one would obtain simply by neglecting the transport noise in (\ref{SPDE for v}). The latter would be $\int_\R C(y)dy$, which is strictly smaller. We can extend this argument by showing the convergence of the higher moments as well. However, the distribution of the stochastic heat equation is not determined by its moments, and  unfortunately, we do not have any proof of the convergence of $\mathcal{V}$ to a solution to the stochastic heat equation. \par{}

We consider a more general setup where we change the balance between the effects of the environment and the molecular diffusivity as we vary $\epsilon$. More precisely, we consider $\mathcal{V}$ solving (\ref{SPDE for v}) with noise $\dot{W}^\epsilon$ having covariance $C^\epsilon(\cdot)= \mu(\epsilon)^2 C(\epsilon^{-1}\cdot)$, where $C$ is a fixed positive definite function, as before. We can also let the molecular diffusivity, $\sigma(\epsilon)$, depend on $\epsilon$ and then $\nu = \sigma(\epsilon)^2+ \mu(\epsilon)^2C(0)$. Two regimes are of interest, in addition to case considered previously, where $\mu(\epsilon)\equiv 1$ and $\sigma(\epsilon)^2=\sigma^2$. The weak environment setting, where $\mu(\epsilon)$ tends to $0$ and $\sigma(\epsilon)$ tends to some $\sigma>0$ as $\epsilon \to 0$; and the weak diffusivity, where $\mu(\epsilon)$ tends to $1$ and $\sigma(\epsilon)$ tends to $0$ as $\epsilon\to 0$.

In the weak environment setting, we will prove the convergence of $\mathcal{V}^\epsilon:= \mathcal{V}^{\epsilon, \lambda(\epsilon)}$ towards the solution to the stochastic heat equation when $\lambda(\epsilon)$ is chosen such that $\lambda(\epsilon) \mu(\epsilon) \epsilon^{\frac{1}{2}}$ converges to a (non-zero) constant. In fact, we can couple the sequence of solutions, $\mathcal{V}$, to (\ref{SPDE for v}) to the limiting solution of the stochastic heat equation. This is achieved by constructing the driving noises from a common space-time white noise, through appropriate mollifications.

Let $W$ be a cylindrical Brownian motion on $L^2(\R)$, and suppose that $C$ can be written as $C=\rho*\rho$ for a symmetric mollifier $\rho$. Denote $C^\epsilon(x-y):= \mu(\epsilon)^2 C(\epsilon^{-1}(x-y))$ and $\rho_\epsilon(y):= \epsilon^{-\frac{1}{2}} \mu(\epsilon) \rho(\epsilon^{-1}y)$, so that $C^\epsilon= \rho_\epsilon*\rho_\epsilon$. Define $W^\epsilon:= \rho_\epsilon*W$, so that $W^\epsilon$ is a centred Gaussian field that is Brownian in time with spatial covariance function $C^\epsilon$. In the following theorem, we will consider solutions to (\ref{SPDE for v}), $\mathcal{V}$, in the weak environment setting where $\mu(\epsilon)$ is taken to be converging to $0$ as $\epsilon$ tends to $0$. Note that whilst we have coupled the solutions, $\mathcal{V}$, they are no longer connected by the diffusive scalings and shifts that we discussed earlier. 
\begin{theorem}\label{Thm:1}
    Suppose that $\mu(\epsilon)= o\left( \log(\epsilon^{-1})^{-\frac{1}{2}} \right)$ and $\lambda(\epsilon)\mu(\epsilon)\epsilon^{\frac{1}{2}}(\int_\R C(y)dy)^{1/2}\to \kappa$ as $\epsilon \to 0$. Let $\mathcal{V}$ be the solution to the SPDE
    \begin{align}\label{SPDE}
        \partial_t \mathcal{V} = \frac{\nu}{2}\partial^2_y \mathcal{V} + \lambda(\epsilon) \mathcal{V} \dot{W}^\epsilon - \partial_y\left(\mathcal{V} \dot{W}^\epsilon\right), \quad \mathcal{V}(0,\cdot)=\delta_0,
    \end{align}
    where $\nu>0$ is constant. Then, for any $t>0$,  as $\epsilon$ tends to zero, the following convergence holds in $L^2(\p \times \R)$
    \begin{align*}
        \mathcal{V}(t) \rightarrow \mathcal{Z}(t),
    \end{align*}
    where $\mathcal{Z}\in C((0,T);C(\R))$ denotes  the solution to the stochastic heat equation,
    \begin{align}
        \partial_t \mathcal{Z} = \frac{\nu}{2}\partial^2_y \mathcal{Z} + \kappa \mathcal{Z} \dot{W}, \quad \mathcal{Z}(0,\cdot)=\delta_0.
    \end{align}
\end{theorem}

The convergence in $L^2(\p \times \R)$, established in Theorem \ref{Thm:1}, implies convergence in distribution of $\mathcal{V}(t)$ as random elements of $L^2(\R)$. If we adjust the definition of ${\mathcal U}$ to account for the weak environment ( noting this makes $\mathcal{U}$ depend on $\epsilon$)   then   \eqref{Def: Tilted kernel}   holds as an equality in distribution,  and  we have the following corollary.
\begin{corollary}
Suppose that ${\mathcal U}$ solves 
\[
 \partial_t \mathcal{U} = \frac{\nu}{2} \partial^2_y \mathcal{U} - \mu(\epsilon) \partial_y\left(\mathcal{U} \dot{W}_C \right),\quad \mathcal{U}(0,\cdot)=\delta_0,
 \]
 where $\mu(\epsilon)= o\left( \log(\epsilon^{-1})^{-\frac{1}{2}} \right)$. Let $\lambda(\epsilon)$ be chosen so that $\lambda(\epsilon)\mu(\epsilon)\epsilon^{\frac{1}{2}}\to \kappa \left(\int_\R C(y)dy\right)^{-1/2}$ as $\epsilon \to 0$. Then, 
 \[
y\mapsto \epsilon^{-1} \exp\bigl({\lambda(\epsilon)^2\nu t/2 + \lambda(\epsilon)y }\bigr) \mathcal{U}\bigl( \epsilon^{-2}t, \epsilon^{-1}(y+\lambda(\epsilon) \nu  t)\bigr)
 \]
 converges in distribution as a random  element of $L^2({\mathbb R})$ to $\mathcal{Z}(t)$, where ${\mathcal Z}$ solves the SHE with parameters $\nu$ and $\kappa$.
\end{corollary}

This extends simply to multiple times, and it is natural to expect that in fact $\bigl(\mathcal{V}(t)\bigr)_{t > 0}$ is converging as a process. We have not established the tightness needed to prove this however. 

Turning to the weak diffusivity regime, we can once again argue at the level of the moments and  show that if we set $\mu(\epsilon)=1$ and suppose that $\sigma(\epsilon)$ tends to $0$ slowly enough that $\sigma(\epsilon)\epsilon^{-1}$ tends to infinity as $\epsilon$ tends to $0$, and we choose $\lambda(\epsilon)$ such that $\lambda(\epsilon) \sigma(\epsilon)^{-\frac{1}{2}} \epsilon^{\frac{1}{2}}$ is converging to a positive constant, $c$, then the second moment of $\mathcal{V}$ converges  to (\ref{second moment limit}), this time  with $\kappa^2 = \sqrt{2} c  \nu \frac{ C(0)}{|C''(0)|^{1/2}}$. We will justify this prediction for $\kappa$ at the end of Section \ref{2pt The general case}. The nature of this convergence is expected to be quite different from in the weak environment setting. Firstly, some smoothing in the spatial variable will be needed to control  oscillations in $ \mathcal{V}$. Secondly,  if the pair  $\bigl( \mathcal{V}, W^\epsilon\bigr)$   converges in distribution (possibly down some subsequence) then the limit of the noises $W^\epsilon$  is not the noise driving the limiting SHE.  


We summarize our  predictions for  general choices of $\mu(\epsilon), \lambda(\epsilon)$ and $\sigma(\epsilon)$ in Figure \ref{diagram}. Let
\begin{align*}
    &\beta = -\lim_{\epsilon\to 0} \frac{\log(\lambda(\epsilon))}{\log(\epsilon)},\\
    &\alpha = -\lim_{\epsilon \to 0} \frac{\log(\mu(\epsilon) \sigma(\epsilon)^{-1})}{\log(\epsilon)},
\end{align*}
where we suppose that the above limits exist and that $\beta$ is positive.
\begin{figure}[htbp]
    \centerline{\includegraphics[scale=0.7]{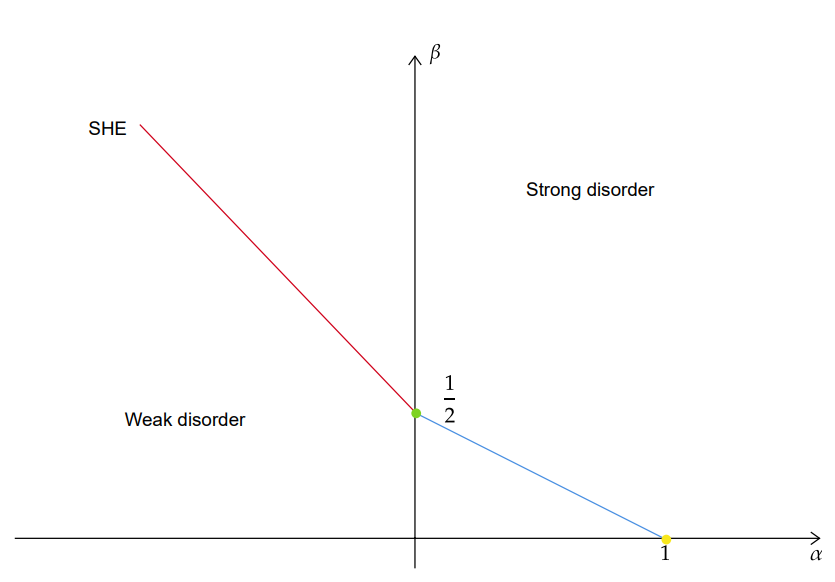}}
    \caption{  The limiting behaviour of the   kernel $\mathcal{V}$ for varying values of the exponents $\alpha$ and $\beta$. The convergence to the SHE proven in Theorem \ref{Thm:1}  corresponds the line segment $\beta = \frac{1}{2}- \alpha$ in the left quadrant; convergence to the SHE is conjectured for the line segment $\beta = \frac{1}{2}(1-\alpha)$ in the right quadrant, including  at  the point $(0, \tfrac{1}{2})$  which corresponds to  diffusive scaling of ${\mathcal U}$.   Below these line segments is the weak disorder regime in which  the kernels converge to the  deterministic  heat equation (Propositions \ref{Weak convergence to the heat equation} and \ref{Thm: HE limit}). The limit of the kernels  is a sticky flow at $(1,0)$ ; for $\alpha>1$ and  $\beta =0$ the limit  is the Arratia flow of coalescing Brownian motions. Above the line segments lies the strong disorder regime in which the limit of the kernels is conjectured to be zero; partially proven in  Proposition \ref{Above the line}. }
    \label{diagram}
\end{figure}

The next two results concern the behaviour of $\mathcal{V}$ in the  weak disorder regime, which in Figure \ref{diagram} corresponds to the region below the line segments.
\begin{proposition}\label{Weak convergence to the heat equation}
    Suppose we are in the setting of the preamble to Theorem \ref{Thm:1}. Suppose further that $\mu(\epsilon)$ is bounded and $\lambda(\epsilon) \mu(\epsilon) \epsilon^{\frac{1}{2}}\sigma^{-\tfrac{1}{2}}_\epsilon$ tends to $0$ as $\epsilon\to 0$. For any $t>0$ and $f\in C^2_b(\R)$ we have
    \begin{align*}
        \int_\R \mathcal{V}(t,y) f(y)dy \to \int_\R p_t(y) f(y)dy \ \text{in } L^2(\p), \ \text{as } \epsilon\to 0.
    \end{align*}
    As previously, $p_t(y)= \frac{1}{\sqrt{2\pi\nu t}} e^{-\frac{y^2}{2\nu t}}$ denotes the density of the heat kernel.
\end{proposition}
Notice that the above result involves  spatial averaging; we expect that there are random fluctuations  of the type described by Dunlap and Gu, \cite{QuenchedLocalLimitStochasticFlows}, which mean that pointwise convergence fails. However, in the weak environment setting these fluctuations vanish in the limit, and we obtain the following result.
\begin{proposition}\label{Thm: HE limit}
    Suppose we are in the setting of the preamble to Theorem \ref{Thm:1}. Suppose further that $\mu(\epsilon) = o((\log(\epsilon^{-1}))^{-\frac{1}{2}})$ and $\lambda(\epsilon)\mu(\epsilon)\epsilon^{\frac{1}{2}} \to 0$ as $\epsilon \to 0$. Then, for each $t>0$, the following convergence holds in $L^2(\p\times \R)$
    \begin{align*}
        \mathcal{V}(t) \to p_t, \text{ as } \epsilon\to 0.
    \end{align*}
\end{proposition}

In the  strong disorder regime, corresponding to  the region above the line segments in Figure \ref{diagram}, the asymptotic behaviour of the kernels is expected to be  quite different. We give  a first result in this direction, showing convergence to zero;   however the argument is  only valid when  $\beta >\frac{1}{2}$, so that there is a region  in Figure \ref{diagram} on the right which is not covered. The proof is an adaptation of an argument from Liggett \cite{liggett2004interacting}, see also \cite{DPinREpathlocalizationandstrongdisorder}.
\begin{proposition}\label{Above the line}
    Suppose we are in the setting of the preamble to Theorem \ref{Thm:1}. Suppose further that $\mu(\epsilon)$ is bounded and $\lambda(\epsilon) \mu(\epsilon) \epsilon^{\frac{1}{2}}$ tends to infnity as $\epsilon\to 0$. Then, for each $t>0$, the following convergence holds in probability
    \begin{align*}
        \int_\R \mathcal{V}(t,y) dy \to 0, \text{ as } \epsilon\to 0.
    \end{align*}
\end{proposition}

\section{The 2-Point Motion}
\subsection{The general case}\label{2pt The general case}
Before we move on to the main part of the proof of Theorem \ref*{Thm:1}, we will show that the second moments of $\mathcal{V}$ can be written in terms of a two-dimensional diffusion, which we will call the two-point motion associated to $\mathcal{V}$.
\begin{proposition}\label{2pt motion}
    Suppose that $\mathcal{V}^{\epsilon, \lambda}$ is the solution to (\ref{SPDE for v}). Let $Y = (Y^1,Y^2)$ be the diffusion in $\R^2$ with generator
    \begin{align*}
        \mathcal{G}f= \frac{1}{2}\sum_{i,j=1}^2 (\sigma^2 \delta_{i,j} + C^\epsilon(y_i-y_j))\scndmxPartial{f}{y_i}{y_j} + \sum_{i=1}^2 \lambda C^\epsilon(y_1-y_2)\Partial{f}{y_i},
    \end{align*}
    where $\delta_{i,j}$ is the Kronecker delta. Then we have the equality
    \begin{align}
        &\E\left[\int_{\R^2} \mathcal{V}^{\epsilon, \lambda}(t,y_1)\mathcal{V}^{\epsilon, \lambda}(t,y_2)f(y_1,y_2)dy \right] = \E\left[ e^{\lambda^2 \int_0^t C^\epsilon(Y^1(s)-Y^2(s))ds} f(Y^1(t),Y^2(t)) \right].\label{generator of Y}
    \end{align}
\end{proposition}
\begin{proof}
    As a consequence of (\ref{Def: Tilted kernel}) and the definition of the underlying flow of kernels, (\ref{Def: Flow of kernels for turbulent SDE}), we have
    \begin{align*}
        &\E\left[\int_{\R^2} \mathcal{V}^{\epsilon, \lambda}(t,y_1)\mathcal{V}^{\epsilon, \lambda}(t,y_2)f(y_1,y_2)dy \right]\\
        =& \E\left[ \left(\int_\R e^{-\frac{\nu}{2} \lambda^2(t-s) +\lambda y} \mathcal{U}^\epsilon (0,t, 0, y) f(y-\lambda \nu t) dy \right)^2 \right]\\
        =& \E\left[ e^{-\frac{\nu}{2} \lambda^2(t-s) +\lambda (X^1(t) + X^2(t))} f(X^1(t)-\lambda \nu t) f(X^2(t) - \lambda \nu t) \right],
    \end{align*}
    where the process $(X(t))_{t\geq 0}$ has the generator given in line (\ref{X generator}). The stated equality follows by applying Girsanov's theorem with change of measure given by the Dol\'eans-Dade exponential of the process $(X^1+X^2)$, and then defining the process $(Y(t))_{t>0}$ as $Y^i(t)= X^i(t)- \lambda \nu t$, for $i=1,2$.
\end{proof}
Note that both $Y^1$ and $Y^2$ are distributed as Brownian motions on $\R$ with diffusivity $\nu= \sigma^2+ C^{\epsilon}(0)$ plus an additional drift, which is the same for both processes. In fact, as $\epsilon\to 0$ the process $Y$ converges in distribution to two-dimensional Brownian motion, which we will now prove. We will then use this result to show that $\mathcal{V}$ converges in a weak sense towards the heat kernel.
\begin{proposition}\label{convergence of the 2pt motion}
    Suppose that $Y$ is the diffusion with the generator given in Line (\ref{generator of Y}) for parameters depending on $\epsilon$, such that $\lambda(\epsilon) \mu(\epsilon) \epsilon^{\frac{1}{2}}\sigma(\epsilon)^{-\frac{1}{2}}$ is bounded in $\epsilon$ (i.e. we are on or below the line in Figure \ref{diagram}) and $\sigma(\epsilon)^2+ C^\epsilon(0)=\nu$ for all $\epsilon>0$. Then, $Y$ converges in distribution to a Brownian motion with diffusivity $\nu$ on $\R^2$ as $\epsilon \to 0$.
\end{proposition}
\begin{proof}
    We apply \cite[Theorem 8.2]{ethier2009markov}; this proves the convergence of finite dimensional distributions by considering the behaviour of the generator. We will use the following estimate, which we derive below. There is a constant $C'>0$ depending only on $t>0$ and $\nu$ such that
    \begin{align}\label{occupation times estimate}
        \E\left[ \int_0^t C^\epsilon(Y^1(s)-Y^2(s))ds^k \right]^{1/k} \leq C'' \lambda(\epsilon)^{-2}.
    \end{align}
    For any $k\in \N$, the occupation times formula gives
    \begin{align*}
        \E\left[ \int_0^t C^\epsilon(Y^1(s)-Y^2(s))ds^k \right] = \E\left[ \left(\int_\R \frac{C^\epsilon(y) }{\sigma(\epsilon)^2 + C^\epsilon(0)- C^\epsilon(y)}\mathcal{L}^y_t(Y^1-Y^2)dy\right)^k\right],
    \end{align*}
    where $\mathcal{L}$ is the local time. If we apply Jensen's inequality to the right hand side, we get 
    \begin{align*}
        \E\left[ \int_0^t C^\epsilon(Y^1(s)-Y^2(s))ds^k \right] \leq m_\epsilon^{k-1} \int_\R \frac{C^\epsilon(y) }{\sigma(\epsilon)^2 + C^\epsilon(0)- C^\epsilon(y)} \E\left[ \mathcal{L}^y_t(Y^1-Y^2)^k \right]dy,
    \end{align*}
    where $m_\epsilon = \int_\R \frac{C^\epsilon(y) }{\sigma(\epsilon)^2 + C^\epsilon(0)- C^\epsilon(y)} dy$.
    One can easily show that for each $t>0$, there is a constant $C>0$, depending only on $\nu$, such that, for any $y\in \R$, $\E[\mathcal{L}^y_t(Y^1-Y^2)^k ]^{1/k} \leq C$. Thus, we have that for all $\epsilon>0$
    \begin{align*}
        \E\left[ \int_0^t C^\epsilon(Y^1(s)-Y^2(s))ds^k \right]^{\frac{1}{k}} &\leq Cm_\epsilon\\
                                                                            &\leq C' \frac{\epsilon \mu(\epsilon)^2}{\sigma(\epsilon)}.
    \end{align*}
    In the above inequality, $C'>0$ is a new constant that depends only on $t$ and $\nu$; the second inequality is a consequence of the assumption that $C^\epsilon=\rho_\epsilon*\rho_\epsilon$, where $\rho_\epsilon= \mu(\epsilon)^2 \epsilon^{-\frac{1}{2}}\rho(\epsilon^{-1}\cdot)$ for $\rho\in C^\infty_c(\R)$. In particular, this means that the main contribution to the integral comes from $y$ small enough that $C^\epsilon(0)- C^\epsilon(y)$ looks quadratic. Our assumption that $\lambda(\epsilon) \mu(\epsilon) \epsilon^{\frac{1}{2}}\sigma(\epsilon)^{-\frac{1}{2}}$ implies (\ref{occupation times estimate}).
    
    Estimate (\ref{occupation times estimate}) allows us to apply \cite[Theorem 8.2]{ethier2009markov} by controlling the difference between the action of the generator (\ref{generator of Y}) and the Laplacian. The Theorem gives us convergence of the finite dimensional distributions. The proof is concluded by noting that since, for any $\epsilon>0$, $Y^1$ and $Y^2$ are both marginally Brownian motions with diffusivity $\nu$ the sequence of processes $(Y)_{\epsilon>0}$ is tight.
\end{proof}
Proposition \ref{Weak convergence to the heat equation}, i.e. $\mathcal{V}(t)$ converges to the heat kernel in a weak sense, is an immediate consequence of the above proposition.
\begin{proof}[Proof of Proposition \ref{Weak convergence to the heat equation}]
    First note that Proposition \ref{2pt motion} implies
    \begin{align*}
        &\E\left[ \int_\R\left( \mathcal{V}(t,y) - p_t(y)\right) f(y)dy ^2 \right] \\
        =& \E\left[ e^{\lambda(\epsilon)^2 \int_0^t C^\epsilon(Y^1(s)-Y^2(s))ds} f(Y^1(t))f(Y^2(t)) \right] - \int_\R p_t(y)f(y) dy^2.
    \end{align*}
    Note that (\ref{occupation times estimate}) implies that the exponential term is converging to $1$ in $L^p(\p)$, for any $p\in [1, \infty)$. Hence, Proposition \ref{convergence of the 2pt motion} shows the above expression vanishes as $\epsilon\to 0$, and thus proves the statement.
\end{proof}
Moving the discussion to the weak diffusivity regime, we will explain the predictions we made in the introduction for the coefficients of the limiting stochastic heat equation. Recall that in the weak diffusivity regime, we assume that $\sigma(\epsilon)$ tends to $0$ more slowly than $\epsilon$, as $\epsilon$ tends to $0$. We also assume that $\frac{\lambda(\epsilon)^2 \epsilon}{\sigma(\epsilon)}$ converges to a positive constant $c$. The occupations times formula gives
\begin{multline}\label{occtimes}
    \lambda(\epsilon)^2 \int_0^t C^\epsilon(Y^1(s)-Y^2(s))ds =  \\
    \frac{\lambda(\epsilon)^2 \epsilon}{2\sigma(\epsilon)} \int_\R \frac{C(\sigma(\epsilon) y)}{1+ \underbrace{\sigma(\epsilon)^{-2}(C(0)-C(\sigma(\epsilon) y))}_{=-\frac{1}{2}C''(\xi_\epsilon(y)) y^2}} \mathcal{L}^{\sigma(\epsilon) \epsilon y}_t(Y^1-Y^2) dy,
\end{multline}
where the equality under the brace is due to Taylor's expansion and $|\xi_\epsilon(y)| \leq \sigma(\epsilon) |y|$. Proposition \ref{convergence of the 2pt motion} shows that the process $(Y^1, Y^2)$ converges in distribution to a Brownian motion $(B^1, B^2)$; if we assume that the local times, $\mathcal{L}^y_t(Y^1-Y^2)$, are converging jointly with $(Y^1,Y^2)$ to $\mathcal{L}^y_t(B^1-B^2)$, then (\ref{occtimes}) suggests the following convergence should hold as $\epsilon \to 0$
\begin{align*}
    &\E\left[ e^{\lambda(\epsilon)^2 \int_0^t C^\epsilon(Y^1(s)-Y^2(s))ds} f(Y^1(t))f(Y^2(t)) \right]\\
    \rightarrow \quad &\E[\exp\left( c \sqrt{2 } |C''(0)|^{-\frac{1}{2}} C(0)  \mathcal{L}^0_t(B^1-B^2) \right)f(B^1(t))f(B^2(t)) ].
\end{align*}
This is  the form of the second moment for the stochastic heat equation,  with the coefficients as predicted in the Introduction. This convergence can be made rigorous, see the thesis \cite{BMsInREs} for details.
\subsection{In a Weak Environment}
In the weak environment setting, we can make stronger statements about the limiting behaviour of the $\mathcal{V}$ by looking at the behaviour of the density of the law of the two-point motion. In fact, we will only need to look at the difference between the two-point motions, which is a one-dimensional diffusion. One explanation for the better behaviour in this setting is that the diffusion coefficient in the generator of this difference is converging to a constant function pointwise, rather than only pointwise almost everywhere, as happens in the other two regimes. As a consequence of Proposition \ref{2pt motion}, we can make the following statement about the process $D:=Y^1-Y^2$.
\begin{proposition}
    The process $D= Y^1-Y^2$ is a diffusion with generator
    \begin{align}\label{generator of D}
        (\sigma^2 + (C^\epsilon(0)-C^\epsilon(\cdot)))\frac{d^2}{dx^2}.
    \end{align}
\end{proposition}
Let $q$ be the transition density of $D$, that is let $\p(D_t\in dy| \ D_s= x) = q(t-s;x,y)$. The existence of this density is standard, see for example \cite{borodin2017stochastic}. In the following, we collect some necessary estimates on $q$, which hold uniformly in $\epsilon > 0$. In the rest of this section, we'll fix a constant $T>0$ and deal with a finite time window $(0,T)$.
\begin{proposition}\label{estimate on q}
    Suppose that $\mu(\epsilon)$ is a bounded sequence, there are constants $c,C>0$ such that for any $\epsilon>0$, $t\in (0,T)$, and $x,y \in \R$
    \begin{align*}
        q(t;x,y) \leq Ct^{-\frac{1}{2}}e^{-c\frac{(x-y)^2}{2t}}.
    \end{align*}
\end{proposition}
\begin{proof}
    We aim to apply the result \cite[Theorem 1]{AronsonBoundsforFunds}, which provides a Gaussian estimate for fundamental solutions to elliptic PDEs. To begin, we will show that $q$ satisfies a PDE. Denote $a_\epsilon(y):= (\sigma^2 + C^\epsilon(0)-C^\epsilon(y))$, and define $F_\epsilon(y):= \int_0^y \frac{1}{a_\epsilon(x)}dx$. Note that for each $\epsilon>0$, $F_\epsilon$ is a continuous and strictly increasing function, and so has a well defined inverse. Let $\Delta$ be the process defined by $\Delta(t):=F_\epsilon(D(t))$, then it is a diffusion with generator $\frac{d}{dx}\left( (a_\epsilon\circ F_\epsilon^{-1})^{-\frac{1}{2}} \frac{d}{dx}\right)$. If we denote the transition density of $\Delta$ by $\tilde{q}$, then since $a$ is bounded above and below uniformly over all $\epsilon>0$, we can apply the result in \cite[Theorem 1]{AronsonBoundsforFunds} to get that there are constants $c, C>0$ such that for all $\epsilon>0$, $t\in (0,T)$, and $x,y \in \R$
    \begin{align*}
        \tilde{q}(t;x,y) \leq Ct^{-\frac{1}{2}}e^{-c\frac{(x-y)^2}{2t}}.
    \end{align*}
    However, it is easy to see that $q(t;x,y)=\tilde{q}(t;F_\epsilon(x),F_\epsilon(y))/a_\epsilon(y)$, so that we in fact have the estimate
    \begin{align}
        q(t;x,y) \leq C a(y)^{-1}t^{-\frac{1}{2}}e^{-c\frac{(F_\epsilon(x)-F_\epsilon(y))^2}{2t}}.
    \end{align}
    The claimed estimate then follows from the fact that $0< \sigma^2 \leq a_\epsilon(y) \leq \nu$ for all $\epsilon>0$ and $y\in \R$. In particular, the claim follows after noting
    \begin{align*}
        |F_\epsilon(y)-F_\epsilon(x)| \geq \nu^{-1}|x-y| \quad \text{and } |a_\epsilon(y)|^{-1}\leq \sigma^{-2}.
    \end{align*}
\end{proof}
The above estimate can be used to prove that $\lambda(\epsilon) \mathcal{V}*\rho_\epsilon$ converges to a multiple of the stochastic heat equation in the weak environment setting. In order to study $\mathcal{V}$ itself, we must look at $q^{\lambda(\epsilon)}$, which we define through the following equality, which we impose for any $f\in C_b(\R)$
\begin{align}\label{defining q kappa}
    \E\left[ \int_{\R^2} \mathcal{V}(t,y_1) \mathcal{V}(t,y_2) f(y_1 - y_2)dy\right]= \int_\R q^{\lambda(\epsilon)}(t,y) f(y) dy.
\end{align}
We can also rewrite the expectation on the left-hand side in terms of the difference of the two-point motion, $D= Y^1-Y^2$. This yields the equality
\begin{align*}
    \E\left[e^{\lambda(\epsilon)^2 \int_0^t C^\epsilon(D(s))ds} f(D(t))\right]= \int_\R q^{\lambda(\epsilon)}(t,y) f(y) dy.
\end{align*}
From which it is easily deduced that $q^0(t,y)=q(t;0,y)$. Suppose that $f\in C^2_b(\R)$, we can apply It\^o's formula, which shows us that $q^{\lambda(\epsilon)}$ satisfies the following PDE
\begin{align}\label{PDE for q rho}
    \partial_t q^{\lambda(\epsilon)} = \partial_y^2 \left(a_\epsilon \cdot q^{\lambda(\epsilon)}\right) + \lambda(\epsilon)^2 C^\epsilon q^{\lambda(\epsilon)}, \ q^{\lambda(\epsilon)}(0, y)= \delta_0(y).
\end{align}
As before, the function $a_\epsilon:\R\to \R$ is defined by
\begin{align*}
    a_\epsilon(y) =\sigma^2 + C^\epsilon(0)- C^\epsilon(y).
\end{align*}
To simplify some notation we will introduce notation for the mass of $\lambda(\epsilon) \rho_\epsilon$, which in most of the cases we consider we will assume to be convergent as $\epsilon$ tends to $0$.
\begin{align*}
    \kappa_\epsilon:= \lambda(\epsilon) \int_\R \rho_\epsilon(y)dy.
\end{align*}
We can estimate the difference between $\mathcal{V}$ and its smoothing $\kappa_\epsilon^{-1}\lambda(\epsilon)\mathcal{V}*\rho_\epsilon$ in terms of $q^{\lambda(\epsilon)}$ as follows.
\begin{align}
    \E&\left[\|\mathcal{V}(t)- \kappa_\epsilon^{-1}\lambda(\epsilon)\mathcal{V}(t)*\rho_\epsilon\|_2^2\right]\nonumber \\
    =& q^{\lambda(\epsilon)}(t,0) - 2\kappa_\epsilon^{-1}\lambda(\epsilon)  \int_\R q^{\lambda(\epsilon)}(t, y) \rho_\epsilon(y) dy + \kappa_\epsilon^{-2}\lambda(\epsilon)^2\int_\R q^{\lambda(\epsilon)}(t,y) \rho_\epsilon*\rho_\epsilon(y)dy\nonumber\\
    =& 2 \int_\R (q^{\lambda(\epsilon)}(t,0) - q^{\lambda(\epsilon)}(t, \epsilon y) )\rho(y) dy + \int_\R (q^{\lambda(\epsilon)}(t,\epsilon y) - q^{\lambda(\epsilon)}(t,0)) \rho*\rho(y)dy.\label{UC: error bound}
\end{align}
Thus, we just need equicontinuity of $q^{\lambda(\epsilon)}(t, \cdot)$, as a sequence of functions depending on $\epsilon$, to get the desired upgrade for the convergence. We will begin proving this by considering the ${\lambda(\epsilon)}=0$ case.
\begin{proposition}\label{q is equicontinuous}
    Suppose that $\mu(\epsilon)^2$ converges to $0$ as $\epsilon$ tends to $0$, then for any $\delta>0$ the sequence of functions formed by $q$ is equicontinuous on $[\delta^2, \infty)\times\R\times\R$.
\end{proposition}
\begin{proof}
    From inequality (II.1.10) proved in \cite[Theorem II.3.1]{Stroock1988} and the fact that $\sigma^2\leq a_\epsilon(y)\leq \nu$, we get that $\tilde{q}$ forms an equicontinuous sequence of functions on $[\delta^2, \infty)\times\R\times\R$. $(F_\epsilon)_{\epsilon>0}$ is equicontinuous because of the upper and lower bounds on $a_\epsilon$, whereas $(a_\epsilon)_{\epsilon>0}$ is equicontinous because of our assumption that $\mu(\epsilon)$ is null. The lower bounds on $a_\epsilon$ then imply that $(a_\epsilon^{-1})_{\epsilon>0}$ is equicontinuous. Thus, the equality $q(t;x,y)=\tilde{q}(t;x,F_\epsilon(y))/a_\epsilon(y)$ shows that $q$ is also equicontinuous.
\end{proof}
We will use this result to prove the equicontinuity for ${\lambda(\epsilon)}\neq 0$, but first we need a pointwise estimate on $q^{\lambda(\epsilon)}$.

\begin{proposition}\label{Upper bound on q kappa}
    Suppose that $\mu(\epsilon)$ and $\kappa_\epsilon$ are bounded for $\epsilon>0$, then there is a constant $C>0$ such that for all $t\in(0,T)$, $y\in \R$ and $\epsilon>0$
    \begin{align*}
        \|q^{\lambda(\epsilon)}(t)\|_\infty\leq Ct^{-\frac{1}{2}}e^{Ct}.
    \end{align*}
\end{proposition}
\begin{proof}
    In order to get the desired estimate, we use Duhamel's principle to get an integral equation for $q^\kappa$ in terms of $q$. In particular, Duhamel's principle tells us that
    \begin{align}\label{Duhamel principle for q}
        q^{\lambda(\epsilon)}(t,y) = q(t; 0, y) + \lambda(\epsilon)^2 \int_0^t \int_\R q(t-s; x, y) C^\epsilon(x) q^{\lambda(\epsilon)}(s,x)dx ds.
    \end{align}
    One can easily check, assuming all the involved functions are sufficiently nice, that the solution to the above integral equation is also a solution to (\ref{PDE for q rho}). One can also verify the above equation by applying It\^o's formula in the definition of $q^{\lambda(\epsilon)}$ with a suitable test function, and then taking a limit. We will use this equality to derive an inequality to which we can apply Gronwall's inequality. \par{}
    Applying (\ref{estimate on q}), we get that
    \begin{align}\label{pointwise inequality on q with potential}
        |q^{\lambda(\epsilon)}(t,y)| \leq C t^{-\frac{1}{2}} e^{-c\frac{y^2}{2t}} + \lambda(\epsilon)^2\int_0^t \int_\R C (t-s)^{-\frac{1}{2}} e^{-c\frac{(x-y)^2}{2(t-s)}}C^\epsilon(x) q^{\lambda(\epsilon)}(s,x)dx ds.
    \end{align}
    We can rewrite this in terms of the normalised version of $\rho_\epsilon$ to get
    \begin{align*}
        \|q^{\lambda(\epsilon)}(t)\|_\infty &\leq Ct^{-\frac{1}{2}} + C\int_0^t (t-s)^{-\frac{1}{2}}\|q^{\lambda(\epsilon)}(s)\|_\infty ds.
    \end{align*}
    The inequality follows by estimating $q^{\lambda(\epsilon)}$ by its supremum and noting that $\lambda(\epsilon)^2 \int_\R C^\epsilon(x)dx = \kappa_\epsilon^2$, which is bounded in $\epsilon$. Iterating this inequality, and allowing the constants to change between lines, we get
    \begin{align*}
        \|q^{\lambda(\epsilon)}(t)\|_\infty &\leq Ct^{-\frac{1}{2}} + C + C\int_0^t \int_0^s (t-s)^{-\frac{1}{2}}(s-u)^{-\frac{1}{2}}\|q^{\lambda(\epsilon)}(u)\|_\infty ds du \\
        &=Ct^{-\frac{1}{2}} + C + C\int_0^t \|q^{\lambda(\epsilon)}(u)\|_\infty du.
    \end{align*}
    The proof is concluded by applying Gronwall's inequality and making some additional simplifications. To check we can apply the inequality we just need to show that the integral on the right hand side is always finite. Note that an $L^1$ bound on $q^{\lambda(\epsilon)}$ provides an $L^\infty$ estimate through Equation (\ref{pointwise inequality on q with potential}), but an  $L^1$ bound is easily provided by setting $f(y)=1$ in Equation (\ref{defining q kappa}).
\end{proof}
\begin{proposition}\label{q lambda is equicontinuous}
    Suppose that $\mu(\epsilon)^2$ converges to $0$ as $\epsilon$ tends to $0$, and that $\kappa_\epsilon$ is bounded, then, for each $t\in(0,T)$, the sequence $(q^{\lambda(\epsilon)}(t, \cdot))_{n\in \N}$ is equicontinuous. 
\end{proposition}
\begin{proof}
    We will use Equation (\ref{Duhamel principle for q}) once more. Fix $t>0$ and let $\delta<t$ be arbitrary, then for any $\epsilon >0$, Proposition \ref{q is equicontinuous} tells us that there exists a $\gamma_\delta>0$ such that for all $|x-y|<\gamma_\delta$ we have
    \begin{align}
        |q^\kappa(t,x)-q^\kappa(t,y)| &\leq \epsilon + \lambda(\epsilon)^2 \int_{0}^{t-\delta^2} \int_\R \epsilon C^\epsilon(z) q^{\lambda(\epsilon)}(s,z)dz ds\\
         &\ + \lambda(\epsilon)^2 \int_{t-\delta^2}^{t} \int_\R (q(t-s;z,x)-q(t-s;z,y))C^\epsilon(z) q^{\lambda(\epsilon)}(s,z)dz ds.
    \end{align}
    If we apply Propositions \ref{Upper bound on q kappa} and \ref{estimate on q}, then we see that this is bounded above by
    \begin{align*}
        \epsilon + \left(\epsilon Ct^{\frac{1}{2}}e^{Ct} + C^2e^{2Ct}\int_{t-\delta^2}^t (t-s)^{-\frac{1}{2}}s^{-\frac{1}{2}}ds\right)\underbrace{\lambda(\epsilon)^2 \int_\R C^\epsilon(z)dz}_{=\kappa_\epsilon^2}.
    \end{align*}
    The statement follows by noting that we are allowed to choose $\delta$ to be sufficiently small.
\end{proof}
We conclude this section with a proof of Proposition \ref{Thm: HE limit}, which follows directly from the convergence of $q^{\lambda(\epsilon)}$.
\begin{proof}[Proof of Proposition \ref{Thm: HE limit}]
   First note that
   \begin{align*}
        \E\left[ \|\mathcal{V}(t) - p_t\|_2^2 \right] = q^{\lambda(\epsilon)}(t,0) - p_{2t}(0).
   \end{align*}
   Hence, we only need to prove the convergence of $q^{\lambda(\epsilon)}(t)(0)$ towards $p_{2t}(0)$. We already have equicontinuity and boundedness of $q^{\lambda(\epsilon)}(t)$ from propositions \ref{q lambda is equicontinuous} and \ref{pointwise inequality on q with potential}; thus, the Arzela-Ascoli theorem implies that the sequence $q^{\lambda(\epsilon)}$ is relatively compact in $L^{\infty}([-K,K])$, where $K>0$ is arbitrarily chosen. From Equation (\ref{Duhamel principle for q}), we get the inequality
   \begin{align*}
    \|q^{\lambda(\epsilon)}(t) -q(t; 0, \cdot)\|_{\infty} \leq \kappa_\epsilon^2 \int_0^t \sup_{(x,y)\in \R^2}|q(t-s; x, y)| \|q^{\lambda(\epsilon)}(s)\|_\infty ds.
   \end{align*}
   Applying Propositions \ref{estimate on q} and \ref{pointwise inequality on q with potential}, and using that $\kappa_\epsilon$ is a null sequence, we see that the right hand side vanishes. Finally, we note that propositions \ref{estimate on q} and \ref{q is equicontinuous} prove that, for each $T>\delta>0$, $q$ is relatively compact as a sequence in $\epsilon$ on $L^\infty([\delta^2, T]\times [-K,K]\times [-K,K])$. Due to Proposition \ref{convergence of the 2pt motion}, the process $D$ converges in distribution to a Brownian motion with diffusivity $2\nu$, combining this with the above relative compactness proves uniform convergence on compact sets of $q(t)$ towards $p_{2t}$.
\end{proof}

\section{The Mild Equation}
In the following, we will use the notation for $f,g\in L^2(\R)$
\begin{align}
    &(f,g)_2=\int_\R f(y)g(y)dy, \quad \|f\|_{2} = (f,f)_2^{\frac{1}{2}}, \nonumber\\
    &(f,g)_{2,\rho_\epsilon}= \int_\R (f*\rho_\epsilon)(y) (g*\rho_\epsilon)(y)dy, \quad \|f\|_{2, \rho_\epsilon} = (f,f)_{2,\rho_\epsilon}^{\frac{1}{2}}.
\end{align}
We will also use the notation below for stochastic integrals with respect to $W^\epsilon$.
\begin{align}
    \int_0^t(f(s), W^\epsilon(ds))_2 &:= \int_0^t \int_\R f(s,y) W^\epsilon(ds,y)dy\nonumber\\
    &=\int_0^t \int_\R f(s,\cdot)*\rho_\epsilon(y) W(ds,dy).
\end{align}
The second inequality is a consequence of the coupling between the fields $W^\epsilon$ that we chose in the setup of Theorem \ref*{Thm:1}. Note that we have the isometry
\begin{align}
    \E\left[\left(\int_0^t(f(s), W^\epsilon(ds))_2\right)^2\right] = \int_0^t \|f(s)\|_{2,\rho_\epsilon}^2 ds.
\end{align}
It is important to note here that $\rho_\epsilon$ is not normalised to have unit mass, so that for a fixed function $f\in L^2(\R)$, $\|f\|_{2, \rho_\epsilon}$ converges to $0$ as $\epsilon$ tends to $0$. However, under the choice of $\lambda(\epsilon)$ made in Theorem \ref{Thm:1}, the function $\lambda(\epsilon) \rho_\epsilon$ has mass $\kappa_\epsilon$, which is converging to a positive constant.
We can derive the SPDE for $\mathcal{V}$, (\ref{SPDE for v}), directly from the SPDE for $\mathcal {U}$.
\begin{lemma}\label{Lemma: SPDE for v}
    With $\mathcal{V}$ and $\mathcal{U}^\epsilon$ defined as in Section \ref{The model}, we have that for each $f\in C^\infty_c(\R)$
    \begin{align}\label{SPDE for v weak formulation}
        &\int_\R \mathcal{V}(t, y) f(y) dy - f(x)\\
        =& \frac{\nu}{2}\int_0^t \int_\R \mathcal{V}(s,y) f''(y) dy ds +  \int_0^t \int_\R \mathcal{V}(s,y) \left( \lambda f(y) + f'(y) \right) W^{\epsilon}(ds, y) dy.
    \end{align}
\end{lemma}
\begin{proof}
    Let $f\in C^\infty_c(\R)$ and define $g(t,y):= e^{-\frac{\nu}{2}\lambda^2 t + \lambda(y-x)} f(y-\lambda\nu t)$. By definition, we have $\int_\R \mathcal{V}(t,x, y) f(y)dy = \int_\R \mathcal{U}(0,t, x, y) g(t,y)dy$. Thus, a straightforward calculation using Equation (\ref{Weak Formulation of u SPDE}) and the stochastic and deterministic Fubini theorems leads to the equality
\begin{align*}
    &\int_0^t \int_\R \partial_s g(s,y) \mathcal{U}^\epsilon(0,s, x, y) dy ds + f(x)\\
    =& \int_\R \mathcal{V}(t,y) f(y) dy -\frac{\nu}{2}\int_0^t \int_\R \partial_y^2g(s,y)\mathcal{U}^\epsilon(0,s,x,y)dy ds\\
    &-\int_0^t\int_\R e^{-\frac{\nu}{2}\lambda^2 s + \lambda(y-x)} (\lambda f(y-\lambda\nu s) + f'(y-\lambda \nu s)) \mathcal{U}^\epsilon(0,s,x,y) W^{\epsilon}(ds, y) dy.
\end{align*}
Noticing that $\partial_t g(t,y) = -\frac{\nu}{2}\partial_y^2 g(t,y) - \frac{\nu}{2} e^{-\frac{\nu}{2}\lambda^2 t + \lambda(y-x)} f''(y-\lambda\nu t)$, we get the desired equation with $\mathcal{W}^\epsilon(ds, y+ \lambda\nu t)$ in place of $\mathcal{W}^\epsilon(ds, y)$. As discussed in Section \ref{The model}, we replace $W^\epsilon$ by its shifted version, which has the same distribution, to get the result as stated.

\end{proof}
From Equation (\ref{SPDE for v}), one can derive the following mild equation for $\mathcal{V}$. For any $f\in C^\infty_c(\R)$
\begin{align}\label{Mild form of the SPDE}
    \left(\mathcal{V}(t), f\right)= P^{\nu}_{t}f(x)+ \int_0^t \left( \mathcal{V}(s) \cdot P_{t-s}^{\nu} (\lambda(\epsilon) f + f'), W^\epsilon(ds) \right)_2 .
\end{align}
The solution to the stochastic heat equation,
\begin{align*}
    \partial_t \mathcal{Z}_\epsilon = \frac{\nu}{2}\partial^2_y \mathcal{Z}_\epsilon + \lambda(\epsilon) \mathcal{Z}_\epsilon \dot{W}^\epsilon,
\end{align*}
 satisfies its own mild equation. For any $f\in C^\infty_c(\R)$
\begin{align*}
    \left(\mathcal{Z}_\epsilon(t), f\right)= P^{\nu}_{t}f(x)+  \lambda(\epsilon) \int_0^t \left( \mathcal{Z}_\epsilon(s) \cdot P_{t-s}^{\nu} f, W^\epsilon(ds) \right)_2 .
\end{align*}
We will use these two mild equations to prove Theorem \ref{Thm:1}.
\subsection{Proof of Theorem \ref{Thm:1}}
The proof of Theorem \ref{Thm:1} relies on a Gronwall argument that will yield the following inequality.
\begin{proposition}\label{Prop: Gronwall}
    There are constants $c,C>0$ such that for every $\epsilon\in (0,1)$ and $t>0$
    \begin{align}
        &\E\left[\|\mathcal{V}(t)-\mathcal{Z}_\epsilon(t)\|_{2,\rho_\epsilon}^2\right]\nonumber\\
        \leq& C\lambda(\epsilon)^{-2}t^{-\frac{1}{2}} e^{ct} \left( \epsilon\log(\epsilon^{-1})+  \mu(\epsilon)^2 \log(\epsilon^{-1})\right).\label{Estimate on the difference}
    \end{align}
\end{proposition}
Before the proof of this statement, we will first derive an estimate to which we can apply Gronwall's inequality, and then collect some lemmas to get an estimate of the form in the proposition.

Also note that the  estimate in Proposition \ref{Prop: Gronwall} shows that the $\|\cdot\|_{\rho_\epsilon}$ norm tends to $0$ faster than $\lambda(\epsilon)^2$. as $\epsilon \to 0$. This is necessary for the estimate to be meaningful because of the fact that the mass of $\rho_\epsilon$ is proportional to $\lambda(\epsilon)^{-1}$.
\begin{lemma}
    There is a constant $C>0$ such that for all $\epsilon>0$ and $\eta\in (0,1)$
    \begin{align}
        \E\left[\left\|\mathcal{V}(t)- \mathcal{Z}_\epsilon(t)\right\|^2_{2,\rho_\epsilon} \right]
        \leq&C  \E\left[\int_0^t (t-s)^{-\frac{1}{2}}\|\mathcal{V}(s)-\mathcal{Z}_\epsilon(s)\|_{2,\rho_\epsilon}^2 ds\right]\nonumber\\
        &+ C\lambda(\epsilon)^{-2} \epsilon^{1-\eta} \E\left[ \int_0^t(t-s)^{\frac{1}{2}\eta -1}\|\mathcal{V}(s)-\mathcal{Z}_\epsilon(s)\|_{2}^2 ds\right]\nonumber\\
        &+ \E\left[ \int_0^t \int_\R\left\| \mathcal{V}(s) P^{\nu}_{t-s}(\rho_\epsilon'(\cdot-x))\right\|_{2, \rho_\epsilon}^2 dx ds\right].\label{transport term 2}
    \end{align}
\end{lemma}
\begin{proof}
    From the previous section, for the difference between $\mathcal{Z}_\epsilon$ and $\mathcal{V}$, we have
    \begin{align*}
        \left(\mathcal{V}(t) - \mathcal{Z}_\epsilon(t), f\right)_2=& \lambda(\epsilon) \int_0^t \left( \left( \mathcal{V}(s) - \mathcal{Z}_\epsilon(s) \right)\cdot P_{t-s}^{\nu} f, W^\epsilon(ds) \right)_2 \nonumber\\
        &+ \int_0^t \left( \mathcal{V}(s) \cdot P_{t-s}^{\nu} f', W^\epsilon(ds) \right)_2 .
    \end{align*}
    Thus, we have
    \begin{align}
        &\E\left[  \left(\mathcal{V}(t) - \mathcal{Z}_\epsilon(t), f\right)^2 \right] \label{inequality one}\\
        \leq& \lambda(\epsilon)^2 \E\left[ \int_0^t \left\| (\mathcal{V}(s) - \mathcal{Z}_\epsilon(s))P^{\nu}_{t-s}f\right\|_{2, \rho_\epsilon}^2 ds\right]+ \E\left[ \int_0^t \left\| \mathcal{V}(s) P^{\nu}_{t-s}f'\right\|_{2, \rho_\epsilon}^2 ds\right].
    \end{align}
    This is not yet in a suitable form for Gronwall's inequality. However, if we choose $f=\rho_\epsilon(\cdot-x)=:\rho_{\epsilon,x}$, and then integrate out the $x$ variable, inequality (\ref{inequality one}) becomes
    \begin{align}
        &\E\left[\left\|\mathcal{V}(t)- \mathcal{Z}_\epsilon(t)\right\|^2_{2,\rho_\epsilon} dx \right] \label{A Gronwall Estimate}\\
        \leq& \lambda(\epsilon)^2 \E\left[ \int_0^t \int_{\R}\left\| (\mathcal{V}(s) - \mathcal{Z}_\epsilon(s))P^{\nu}_{t-s}(\rho_{\epsilon, x})\right\|_{2, \rho_\epsilon}^2 dx ds\right]\nonumber\\
        &+ \E\left[ \int_0^t \int_\R\left\| \mathcal{V}(s) P^{\nu}_{t-s}(\rho_{\epsilon, x}')\right\|_{2, \rho_\epsilon}^2 dx ds\right]\nonumber.
    \end{align}
    We will show that the first term after the inequality can be bounded above by
    \begin{align}
        &C \E\left[\int_0^t (t-s)^{-\frac{1}{2}}\|\mathcal{V}(s)-\mathcal{Z}_\epsilon(s)\|_{2,\rho_\epsilon}^2 ds\right]\nonumber\\
        &+ C\lambda(\epsilon)^{-2} \epsilon^{1-\eta} \E\left[ \int_0^t(t-s)^{\frac{1}{2}\eta -1}\|\mathcal{V}(s)-\mathcal{Z}_\epsilon(s)\|_{2}^2 ds\right]\nonumber.
    \end{align}
    We will also show that the second term after the inequality in (\ref{A Gronwall Estimate}) tends to $0$ as $\epsilon\to 0$. This, together with bounds on the $L^2$ norms of $\mathcal{V}$ and $\mathcal{Z}$, will allow us to apply Gronwall's inequality to (\ref{A Gronwall Estimate}) and prove Proposition \ref{Prop: Gronwall}.
    Recall that $C^\epsilon:=\rho_\epsilon*\rho_\epsilon$, so that we may rewrite as follows
    \begin{align}
        &\int_{\R}\left\| (\mathcal{V}(s) - \mathcal{Z}_\epsilon(s))P^{\nu}_{t-s}(\rho_\epsilon(\cdot-x))\right\|_{2, \rho_\epsilon}^2 dx \nonumber\\
        =& \int_{\R^2}\int_\R p^{\nu}_{t-s}*\rho_\epsilon(x-y_1)p^{\nu}_{t-s}*\rho_\epsilon(x-y_2)dx \left(\mathcal{V}(s)-\mathcal{Z}_\epsilon(s)\right)^{\otimes 2}(y_1,y_2) C^\epsilon(y_1-y_2) dy\nonumber\\
        =&\int_{\R^2}(p^{\nu}_{t-s}*\rho_\epsilon)*(p^{\nu}_{t-s}*\rho_\epsilon)(y_1-y_2) \left(\mathcal{V}(s)-\mathcal{Z}_\epsilon(s)\right)^{\otimes 2}(y_1,y_2) C^\epsilon(y_1-y_2) dy\nonumber.
    \end{align}
    By rearranging the convolutions and using the convolution property of the heat kernel, we get that the above expression is equal to
    \begin{align}
        & \int_\R \left(\mathcal{V}(s)-\mathcal{Z}_\epsilon(s)\right)*\left((p^{\nu}_{2(t-s)} *C^\epsilon) \cdot C^\epsilon\right)(y) \left(\mathcal{V}(s,y)-\mathcal{Z}_\epsilon(s,y)\right)dy\nonumber\\
        =&\left( \left(\mathcal{V}(s)-\mathcal{Z}_\epsilon(s)\right)*\left((p^{\nu}_{2(t-s)} *C^\epsilon) \cdot C^\epsilon\right), \mathcal{V}(s)-\mathcal{Z}_\epsilon(s)\right)_2.\label{The convoluter}
    \end{align}
    Young's inequality gives us the estimate
    \begin{align}
        &\left( \left(\mathcal{V}(s)-\mathcal{Z}_\epsilon(s)\right)*\left((p^{\nu}_{2(t-s)} *C^\epsilon-p^{\nu}_{2(t-s)} *C^\epsilon(0)) \cdot C^\epsilon\right), \mathcal{V}(s)-\mathcal{Z}_\epsilon(s)\right)_2\nonumber\\
        \leq& \|\mathcal{V}(s)- \mathcal{Z}_\epsilon(s)\|_2^2 \|(p^{\nu}_{2(t-s)} *C^\epsilon-p^{\nu}_{2(t-s)} *C^\epsilon(0)) \cdot C^\epsilon\|_1.\nonumber
    \end{align}
    For each $\eta\in[0,1]$ there is a $C>0$ such that we have the following inequality
    \begin{align}
        & \|(p^{\nu}_{2(t-s)} *C^\epsilon-p^{\nu}_{2(t-s)} *C^\epsilon(0)) \cdot C^\epsilon\|_1\nonumber\\
        \leq& (2\pi \nu)^{-\frac{1}{2}}\lambda(\epsilon)^{-2}(t-s)^{\frac{1}{2}\eta -1}\int_\R |y|^{1-\eta}C^\epsilon(y)dy\nonumber\\
        \leq& C\epsilon^{1-\eta}\lambda(\epsilon)^{-4} (t-s)^{\frac{1}{2}\eta -1}.\nonumber
    \end{align}
    The last inequality is a consequence of the assumption that $C^\epsilon= \rho_\epsilon*\rho_\epsilon$ for $\rho_\epsilon= \mu(\epsilon)^2 \epsilon^{-\frac{1}{2}} \rho(\epsilon^{-1}\cdot)$, where $\rho\in C^\infty_c(\R)$, and the assumption that $\lambda(\epsilon)^2\mu(\epsilon)^2 \epsilon$ is bounded in $\epsilon$. Another consequence is that $(p^{\nu}_{2(t-s)} *C^\epsilon)(0)\leq C' \lambda(\epsilon)^{-2}(t-s)^{-\frac{1}{2}}$ for some constant $C'>0$. Hence, line (\ref{The convoluter}) is bounded above by
    \begin{align}
        &C' \lambda(\epsilon)^{-2}(t-s)^{-\frac{1}{2}}\|\mathcal{V}(s)-\mathcal{Z}_\epsilon(s)\|_{2,\rho_\epsilon}^2+ C\epsilon^{1-\eta} \lambda(\epsilon)^{-4}(t-s)^{\frac{1}{2}\eta -1}\|\mathcal{V}(s)-\mathcal{Z}_\epsilon(s)\|_{2}^2.\nonumber
    \end{align}
    Applying this to (\ref{A Gronwall Estimate}) yields the desired estimate.
\end{proof}
So that we just need a good estimate on the second and third lines of (\ref{transport term 2}). Then, we can apply Gronwall's inequality and get a meaningful estimate.
Note the that the second line of (\ref{transport term 2}) is bounded above by 
\begin{align}
    2C\lambda(\epsilon)^2 \epsilon^{1-\eta} \int_0^t(t-s)^{\frac{1}{2}\eta -1}\E\left[ \int_\R \mathcal{V}(s,y)^2 + \mathcal{Z}_\epsilon(s,y)^2 dy \right]ds.
\end{align}
In the following lemma, we provide estimates on the $L^2$ norms of $\mathcal{V}$ and $\mathcal{Z}_\epsilon$.
\begin{lemma}\label{L2 estimates}
    There are constants $c,C>0$ such that for all $t>0$
    \begin{align}
        \E\left[ \int_\R \mathcal{V}(t,y)^2 dy  \right] &\leq C t^{-\frac{1}{2}}e^{ct};\nonumber\\
        \E\left[\int_\R \mathcal{Z}_\epsilon(t,y)^2 dy\right]& \leq Ct^{-\frac{1}{2}}e^{ct}\nonumber.
    \end{align}
\end{lemma}
\begin{proof}
    For the second inequality we can square and integrate the mild equation for $\mathcal{Z}_\epsilon$ to get 
    \begin{align}
        &\E\left[ \int_\R \mathcal{Z}_\epsilon(t,y)^2dy \right]\nonumber\\
         \leq& C(t-s)^{-\frac{1}{2}}+ \lambda(\epsilon)^2\int_0^t \int_\R \int_{\R^2} p_{t-s}(x-y_1) p_{t-s}(x-y_1) \mathcal{Z}_\epsilon(s,y_1) \mathcal{Z}_\epsilon(s,y_2)C^\epsilon(y_1-y_2)dy dx ds\nonumber\\
        \leq&  C(t-s)^{-\frac{1}{2}}+ \kappa_\epsilon^2 \int_0^t \int_\R \int_{\R} p_{t-s}(x-y)^2 \mathcal{Z}_\epsilon(s,y)^2 dy dx ds\nonumber\\
        \leq& C(t-s)^{-\frac{1}{2}}+ C \int_0^t \int_{\R} (t-s)^{-\frac{1}{2}} \mathcal{Z}_\epsilon(s,y)^2 dy ds.\nonumber
    \end{align}
    The desired inequality follows by iterating the inequality derived above, and then applying Gronwall's inequality.\par{}
    For the first inequality, we instead estimate in terms of the two point motions through the following equality.
    \begin{align}
        \E\left[ \int_\R \mathcal{V}(t,y)^2 dy  \right] = q^{\lambda(\epsilon)}(t,0),\nonumber
    \end{align}
    where $q^{\epsilon, \lambda(\epsilon)}$ is as defined in (\ref{defining q kappa}). The required estimate follows from Proposition \ref{Upper bound on q kappa}.
\end{proof}
This result allows us to estimate the second line in (\ref{transport term 2}), we now move on to estimating the third line.
\begin{lemma}\label{bound for the third line}
    There is a constant $C>0$ such that for every $n\in \N$ with $n>1$ and $t>0$
    \begin{align*}
        &\E\left[ \int_0^t \int_\R\left\| \mathcal{V}(s) P_{t-s}(\rho_\epsilon'(\cdot-x))\right\|_{2, \rho_\epsilon}^2 dx ds\right] \leq C\lambda(\epsilon)^{-2}t^{\frac{1}{2\log(\epsilon^{-1})}-\frac{1}{2}} e^{ct} \mu(\epsilon)^2 \log(\epsilon^{-1}).
    \end{align*}
\end{lemma}
\begin{proof}

Some simple rearrangements using that $C^\epsilon = \rho_\epsilon*\rho_\epsilon$ give the equality
\begin{align*}
    &\int_\R\left\| \mathcal{V}(s) P_{t-s}(\rho_\epsilon'(\cdot-x))\right\|_{2, \rho_\epsilon}^2 dx\\
    =&\int_{\R^2} (\mathcal{V}(s))^{\otimes 2}(y) \int_{\R^2} p^{\otimes 2 }_{t-s}(y -z) (C^\epsilon) ''(z_1-z_2)dz C^\epsilon(y_1 - y_2) dy.
\end{align*}
We will apply a combination of the following two bounds.
\begin{align*}
    &\int_{\R^2} p^{\otimes 2 }_{t-s}(y -z) (C^\epsilon) ''(z_1-z_2)dz\nonumber\\
    \leq& \begin{cases}
        C\mu(\epsilon)^2 \epsilon^{-2}\\
        C\mu(\epsilon)^2 (t-s)^{-1}.
    \end{cases}
\end{align*}
Thus, we see that the right hand side in the lemma is bounded, for any $\gamma \in [0,1)$, by
\begin{align*}
    &C\E\left[ \int_0^t \int_{\R^2} \mu(\epsilon)^2 \epsilon^{-2\gamma} (t-s)^{\gamma-1} \mathcal{V}(s,y_1)\mathcal{V}(s,y_2)C^\epsilon(y_1-y_2)dy ds\right].
\end{align*}
Which, by the Cauchy-Schwartz inequality, is bounded above by
\begin{align}\label{derivative term}
    &C\mu(\epsilon)^2 \epsilon^{-2\gamma} \underbrace{\int_\R C^\epsilon(y)dy}_{= \kappa_\epsilon^2 \lambda(\epsilon)^{-2}}\int_0^t (t-s)^{\gamma-1}\E\left[ \| \mathcal{V}(s)\|_2^2 \right]ds.
\end{align}

We recall that we assumed $\kappa_\epsilon$ is bounded, so that we can apply Lemma \ref{L2 estimates} to get the following upper bound on (\ref{derivative term}), for a new constant $C>0$.
\begin{align}
    &C\mu(\epsilon)^2 \epsilon^{-2\gamma} \lambda(\epsilon)^{-2} \int_0^t (t-s)^{\gamma-1} s^{-\frac{1}{2}} e^{cs}ds\nonumber\\
    &\leq C t^{\gamma-\frac{1}{2}} e^{ct} \mu(\epsilon)^2 \epsilon^{-2\gamma} \lambda(\epsilon)^{-2} B(\gamma, \tfrac{1}{2}),\label{derivative term estimate 1}
\end{align}
where $B$ is the Beta function. Note that $B(\gamma, \frac{1}{2}) = \frac{\Gamma(\gamma) \Gamma(\frac{1}{2})}{\Gamma(\gamma+\frac{1}{2})}$, which is bounded above for all $\gamma\in(0,1]$ by $C\gamma^{-1}$, for some constant $C>0$. Thus, Line (\ref{derivative term estimate 1}) is bounded above by
\begin{align}
    Ct^{\gamma-\frac{1}{2}} e^{ct} \gamma^{-1} \mu(\epsilon)^2 \epsilon^{-2\gamma} \lambda(\epsilon)^{-2},\nonumber
\end{align}
where, once again, the constant $C>0$ has changed between lines. If we set $\gamma= \frac{1}{2\log(\epsilon^{-1})}$, which we can do for $\epsilon\in(0,1)$, then we get the desired upper bound.
\end{proof}
\begin{proof}[Proof of Proposition \ref{Prop: Gronwall}]
Thus, applying Lemmas \ref{L2 estimates} and \ref{bound for the third line} to (\ref{transport term 2}), we get the inequality (for a new constant $C>0$)
\begin{align}
    &\E\left[ \|\mathcal{V}(t)-\mathcal{Z}_\epsilon(t)\|_{2,\rho_\epsilon}^2 \right]\\
    \leq&C\lambda(\epsilon)^{-2} \E\left[ \int_0^t (t-s)^{-\frac{1}{2}}\|\mathcal{V}(s)-\mathcal{Z}_n(s)\|_{2,\rho_\epsilon}^2 ds\right]\nonumber\\
    &+ C \lambda(\epsilon)^{-2} \epsilon^{1-\eta} \int_0^t (t-s)^{\frac{1}{2}\eta-1} s^{-\frac{1}{2}}e^{cs}ds + C\lambda(\epsilon)^{-2}t^{\frac{1}{2\log(\epsilon^{-1})}-\frac{1}{2}} e^{ct} \mu(\epsilon)^2 \log(\epsilon^{-1})\nonumber\\
    \leq& C\lambda(\epsilon)^{-2} \E\left[ \int_0^t (t-s)^{-\frac{1}{2}}\|\mathcal{V}(s)-\mathcal{Z}_\epsilon(s)\|_{2,\rho_\epsilon}^2 ds\right]\nonumber\\
    &+ C\lambda(\epsilon)^{-2}t^{\frac{1}{2\log(\epsilon^{-1})}-\frac{1}{2}} e^{ct} \left( \epsilon\log(\epsilon^{-1})+  \mu(\epsilon)^2 \log(\epsilon^{-1})\right).\label{almost at Gronwall}
\end{align}
Where for the last inequality, we have repeated our estimate for the Beta function and set $\eta = \frac{1}{\log(\epsilon^{-1})}$. If we apply this inequality, we can see
\begin{align}
    &\E\left[ \int_0^t (t-s)^{-\frac{1}{2}}\|\mathcal{V}(s)-\mathcal{Z}_\epsilon(s)\|_{2,\rho_\epsilon}^2 ds\right]\nonumber\\
    \leq&C\lambda(\epsilon)^{-2}\int_0^t (t-s)^{-\frac{1}{2}} \E\bigg[s^{\frac{1}{2\log(\epsilon^{-1})}-\frac{1}{2}} e^{cs} \left( \epsilon\log(\epsilon^{-1})+  \mu(\epsilon)^2 \log(\epsilon^{-1})\right)\nonumber\\
    &+ \int_0^s(s-u)^{-\frac{1}{2}} \|\mathcal{V}(u)-\mathcal{Z}_n(u)\|_{2,\rho_\epsilon}^2du \bigg]ds\nonumber\\
    \leq&  C \lambda(\epsilon)^{-2} t^{\frac{1}{2\log(\epsilon^{-1})}} e^{ct} \left( \epsilon\log(\epsilon^{-1})+  \mu(\epsilon)^2 \log(\epsilon^{-1})\right) + C \lambda(\epsilon)^{-2} \int_0^t \E\left[\|\mathcal{V}(u)-\mathcal{Z}_\epsilon(u)\|_{2,\rho_\epsilon}^2 \right]ds.\nonumber
\end{align}
Applying this bound to (\ref{almost at Gronwall}), we can apply Gr\"onwall's inequality. Which, after some elementary simplifications, yields the desired inequality.
\end{proof}
\begin{proof}{Proof of Theorem \ref{Thm:1}}
As a consequence of the previous proposition, we have proved that if $\mu(\epsilon)=o\left( \log(\epsilon^{-1})^{-\frac{1}{2}}\right)$ then
\begin{align}
    \lambda(\epsilon)^{2}\E\left[ \|(\mathcal{V}(t)-\mathcal{Z}_\epsilon(t))*\rho_\epsilon\|_2^2 \right]\to 0, \ \text{as } \epsilon\to 0.\nonumber
\end{align}
We also have that for each $t>0$, $\mathcal{Z}_\epsilon(t)\to \mathcal{Z}(t)$ in $L^2(\p\times \R)$, this follows from estimates in \cite{Bertini1995}. Therefore,
\begin{align*}
    \E\left[ \|\mathcal{V}(t)*\hat{\rho}_\epsilon-\mathcal{Z}_\epsilon(t)\|_2^2 \right]\to 0, \ \text{as } \epsilon\to 0,
\end{align*}
where $\hat{\rho}_\epsilon:= \kappa_\epsilon^{-1} \lambda(\epsilon)\rho_\epsilon$ is the normalised version of $\rho_\epsilon$. We now just need to show that the mollified version of $\mathcal{V}$ is close to $\mathcal{V}$ itself.
From the triangle inequality, we have
\begin{align}
    \E\left[\|\mathcal{V}(t)- \mathcal{Z}(t)\|_2^2\right] \leq 2 \E\left[\|\mathcal{V}(t)- \mathcal{V}(t)*\hat{\rho}_\epsilon\|_2^2\right] + 2\E\left[\|\mathcal{V}(t)*\hat{\rho}_\epsilon-\mathcal{Z}(t)\|_2^2\right].\label{UC: Triangle inequality}
\end{align}
Theorem \ref{Thm:1} tells us that the latter term on the right hand side tends to $0$ as $\epsilon\to 0$, and we discussed the former term in the previous section. Recall from (\ref{UC: error bound}) that we only need to control the the following increment of $q^{\lambda(\epsilon)}$ to get the desired convergence.
\begin{align*}
    2 \int_\R (q^{\lambda(\epsilon)}(t,0) - q^{\lambda(\epsilon)}(t, \epsilon y) )\rho(y) dy + \int_\R (q^{\lambda(\epsilon)}(t,\epsilon y) - q^{\lambda(\epsilon)}(t,0)) \rho*\rho(y)dy.
\end{align*}
Since both $\rho$ and $\rho*\rho$ have compact support, we only need to apply Proposition \ref{q lambda is equicontinuous} to finish the proof.
\end{proof}
\section{Proof of Proposition \ref{Above the line}}

Let $\mathcal{V}$ be the solution to the SPDE (\ref{SPDE for v}).
Then $v_t:=\int_\R \mathcal{V}(t,y) dy$ is a martingale with quadratic variation
\[
\lambda^2 \int v(t,x) v(t,y) C^\epsilon(x-y)   \;dxdy.
\]
Applying It\^{o}'s formula and taking expectations gives
\[
\frac{ d }{dt} {\mathbb E} [ v_t^{1/2}]= -\frac{\lambda^2}{4}{\mathbb E}\left [ v_t^{-3/2} \int \mathcal{V}(t,x) \mathcal{V}(t,y) C^\epsilon(x-y)   \;dxdy\right]
\]
Now, similarly to previous sections, assume the covariance can be expressed as 
\[C^\epsilon(x)= \int \rho_\epsilon(z)\rho_\epsilon(x-z) dz 
\]
with $\rho_\epsilon:= \mu(\epsilon)^2 \epsilon^{-\frac{1}{2}} \rho(\epsilon^{-1}\cdot)$ for a positive and symmetric function $\rho:\R\to \R$. Then, for any $a>0$, by the Cauchy-Schwartz inequality,
\[
 \int \mathcal{V}(t,x) \mathcal{V}(t,y) C^\epsilon(x-y)   \;dxdy= \int \bigl((\mathcal{V}(t) * \rho_\epsilon)(z)\bigr)^2  dz \geq    \frac{1}{2a} \left(\int_{-a}^a (\mathcal{V}(t) * \rho_\epsilon)(z) dz\right)^2.
\]
Let $\kappa_\epsilon = \int_\R \rho_\epsilon(y)dy$, so that $\int_{-\infty}^\infty (\mathcal{V}(t) * \rho_\epsilon)(z) dz= \kappa_\epsilon v_t  $, and the right-hand side above can be written as
\[
\frac{1}{2a}\left(\kappa_\epsilon v_t- \int_{{\mathbb R}\setminus[-a,a] }(v_t * \rho_\epsilon)(z) dz \right)^2 
\]
Since $(x-y)^2= x^2(1-y/x)^2\geq x^2( 1-2y/x) \geq x^2(1-2 (y/x)^{1/2}) = x^2-2y^{1/2}x^{3/2}$ whenever $0\leq y\leq x$, we deduce that
\begin{align*}
&v_t^{-3/2} \int \mathcal{V}(t,x) \mathcal{V}(t,y) C^\epsilon(x-y)   \;dxdy  \\ \geq&  \frac{\kappa_\epsilon^2}{2a}v_t^{1/2}-\frac{\kappa_\epsilon^{3/2}}{a}\left( \int_{{\mathbb R}\setminus[-a,a] }(v_t * \rho_\epsilon)(z) dz \right)^{1/2}.
\end{align*}
Taking expectations, and using Jensen's inequality finally gives
\[
\frac{ d }{dt} {\mathbb E} [ v_t^{1/2}] \leq -\frac{\lambda^2 \kappa_\epsilon^2}{8a}{\mathbb E}\bigl[ v_t^{1/2}\bigr] +\frac{\lambda^2 \kappa_\epsilon^{3/2}}{4a}{\mathbb E} \left[ \int_{{\mathbb R}\setminus[-a,a] }(v_t * \rho_\epsilon)(z) dz \right]^{1/2}
\]
Express $\rho_\epsilon$ as $\kappa_\epsilon\hat{\rho}_\epsilon$ where $\hat{\rho}_\epsilon$ has unit mass. 
Then integrating the inequality  gives 
\begin{multline}\label{inequality}
{\mathbb E} [ v_t^{1/2}] \leq \exp\left \{ -\frac{\lambda^2 \kappa_\epsilon^2}{8a} t \right\}  {\mathbb E} [ v_0^{1/2}] \\
+ \frac{\lambda^2 \kappa_\epsilon^2}{4a} \int_0^t ds\; \exp\left \{ -\frac{\lambda^2 \kappa_\epsilon^2}{8a} (t-s) \right\} {\mathbb E} \left[ \int_{{\mathbb R}\setminus[-a,a] }(v_s * \hat{\rho}_\epsilon)(z) dz \right]^{1/2} 
\end{multline}

Now we can move on to the proof of Proposition \ref{Above the line}.
\begin{proof}[Proof of Proposition \ref{Above the line}]
The idea is to use the estimate \eqref{inequality} to deduce that ${\mathbb E}[v_t^{1/2}]$ tends to zero with $\epsilon$. 

First notice that Jensen's inequality and the initial condition implies that
\[
{\mathbb E} [ v_t^{1/2} ]\leq {\mathbb E}[ v_t] ^{1/2} =1.
\]

We can write 
\[{\mathbb E} \left[ \int_{{\mathbb R}\setminus[-a,a] }(v_s(\cdot) * \hat{\rho}_\epsilon)(z) dz \right]
\]
as 
\[
{\mathbb P} \bigl( | \sqrt{s} Z+ \epsilon W |>a \bigr)
\]
where $W$ and $Z$ are independent random variables, the distribution of $W$ having density $\hat{\rho}$ and $Z$ being a Gaussian with variance $\nu$.
Consequently  the  second term on the RHS of \eqref{inequality}  can be made arbitraily small, for all $\epsilon$ by choosing $a$ large enough.

The hypothesis that $\epsilon^{1/2} \lambda(\epsilon) \mu(\epsilon)  \rightarrow \infty$ ensures that $ \lambda(\epsilon)^2 \kappa_\epsilon^2/a \rightarrow \infty$ for any fixed $a$, which in turn implies that, for a fixed $a$, the first term on the RHS of \eqref{inequality} is tending to zero.

\end{proof}

\textbf{Acknowledgement}. D.B. was supported by EPSRC grant No. EP/W522594/1.
\cleardoublepage
\bibliography{bibliography}{}
\bibliographystyle{alpha}
\end{document}